\documentclass[11pt,a4paper]{article}
\usepackage{amssymb}
\usepackage{amsmath}
\usepackage{amsfonts}

\newtheorem{theorem}{Theorem}

\newtheorem{definition}{Definition}
\newtheorem{example}{Example}

\newtheorem{remark}{Remark}

\newenvironment{proof}[1][Proof]{\noindent\textbf{#1.} }{\ \rule{0.0em}{0.0em}}
\input{tcilatex}
\begin{document}

\title{\textbf{Bicomplex Mobius Transformation}}
\author{Chinmay Ghosh \\
%EndAName
Guru Nanak Institute of Technology\\
157/F, Nilgunj Road, Panihati, Sodepur\\
Kolkata-700114, West Bengal, India.\\
E-mail: chinmayarp@gmail.com}
\date{}
\maketitle

\begin{abstract}
{\footnotesize In this article the bicomplex version of Mobius
transformation is defined and special attention is paid to find the fixed
points of a bicomplex Mobius transformation.}

\end{abstract}

\section{\textbf{Introduction}}

In $1843$ William Rowan Hamilton $\left( 1805-1865\right) $ \cite{2}\ gave
the concept of quarternions. In $1882$ Corrado Segre $\left(
1860-1924\right) $ \cite{10} introduced a new number system called bicomplex
numbers. Both the number systems are generalization of complex numbers by
four real numbers but are essentially different in two important ways. The
quarternions form a noncommutative division algebra whereas\ bicomplex
numbers \cite{5} form a commutative ring with divisors of zero. Many
properties of such numbers have been discovered during the last few years.
Many researchers \cite{1},\cite{3},\cite{4},\cite{6},\cite{7},\cite{8} and 
\cite{9}\ are doing their works to find the algebraic, topological
properties of bicomplex numbers.\ Also recently the iteration of functions
bicomplex numbers are of interest of some researchers. In this article the
bicomplex version of Mobius transformation is defined and special attention
is paid to find the fixed points of a bicomplex Mobius transformation.

\section{Preliminary definitions}

The set of complex numbers $\mathbb{C}$ forms an algebra generated by two
real numbers and an imaginary element $\mathbf{i}$\ with the property $%
\mathbf{i}^{2}=-1.$\ Similarly duplicating this process by the complex
numbers and three imaginary elements $\mathbf{i}_{1},\mathbf{i}_{2}$ and $%
\mathbf{j}$ governed by the rules:%
\begin{equation*}
\mathbf{i}_{1}^{2}=\mathbf{i}_{2}^{2}=-1,\mathbf{j}^{2}=1
\end{equation*}%
and%
\begin{eqnarray*}
\ \mathbf{i}_{1}\mathbf{i}_{2} &=&\mathbf{i}_{2}\mathbf{i}_{1}=j \\
\mathbf{i}_{1}\mathbf{j} &=&\mathbf{ji}_{1}=-\mathbf{i}_{2} \\
\mathbf{i}_{2}\mathbf{j} &=&\mathbf{ji}_{2}=-\mathbf{i}_{1}.
\end{eqnarray*}%
one can define the bicomplex numbers.

Let us denote $\mathbb{C}(\mathbf{i}_{1})=x+y\mathbf{i}_{1}:x,y\in \mathbb{R}%
;$ and $\mathbb{C}(\mathbf{i}_{2})=x+y\mathbf{i}_{2}:x,y\in \mathbb{R}.$
Each of $\mathbb{C}(\mathbf{i}_{1})$ and $\mathbb{C}(\mathbf{i}_{2})$\ is
isomorphic to $\mathbb{C}.$

\begin{definition}[Bicomplex numbers]
\cite{6} Bicomplex numbers are defined as%
\begin{equation*}
\mathbb{T}=\{z_{1}+z_{2}\mathbf{i}_{2}:z_{1},z_{2}\in \mathbb{C}(\mathbf{i}%
_{1})\}.
\end{equation*}
\end{definition}

One can identify $\mathbb{C}$\ into $\mathbb{T}$\ in two different ways by
the sets%
\begin{equation*}
\mathbb{C}(\mathbf{i}_{1})=\left\{ z_{1}+0\mathbf{i}_{2}\right\} \subset 
\mathbb{T},
\end{equation*}%
\begin{equation*}
\text{and }\mathbb{C}(\mathbf{i}_{2})=\left\{ z_{1}+z_{2}\mathbf{i}%
_{2}:z_{1},z_{2}\in \mathbb{R}\right\} \subset \mathbb{T}.
\end{equation*}%
\ Although $\mathbb{C}(\mathbf{i}_{1})$ and $\mathbb{C}(\mathbf{i}_{2})$\
are isomorphic to $\mathbb{C}$ they are essentially different in $\mathbb{T}%
. $\ 

\begin{definition}[Duplex numbers]
\cite{6} If we put $z_{1}=x$ and $z_{2}=y\mathbf{i}_{1}$ with $x;y\in 
\mathbb{R}$ in $z_{1}+z_{2}\mathbf{i}_{2}$, then we obtain the following
subalgebra of hyperbolic numbers, also called \textit{duplex numbers},%
\begin{equation*}
\mathbb{D}=\{x+y\mathbf{j}:\mathbf{j}^{2}=1,x;y\in \mathbb{R}\}.
\end{equation*}
\end{definition}

\begin{definition}[Bicomplex conjugates]
\cite{6} For bicomplex numbers, there are three possible conjugations. Let $%
w\in \mathbb{T}$ and $z_{1},z_{2}\in \mathbb{C}$ such that $%
w=z_{1}+z_{2}i_{2}$. Then the three conjugations are defined as:%
\begin{eqnarray*}
w^{\dag _{1}} &=&(z_{1}+z_{2}\mathbf{i}_{2}\mathbf{)}^{\dag _{1}}\mathbf{=}%
\bar{z}_{1}+\bar{z}_{2}\mathbf{i}_{2}, \\
w^{\dag _{2}} &=&(z_{1}+z_{2}\mathbf{i}_{2}\mathbf{)}^{\dag _{2}}\mathbf{=}%
z_{1}-z_{2}\mathbf{i}_{2}, \\
w^{\dag _{3}} &=&(z_{1}+z_{2}\mathbf{i}_{2}\mathbf{)}^{\dag _{3}}\mathbf{=}%
\bar{z}_{1}-\bar{z}_{2}\mathbf{i}_{2}.
\end{eqnarray*}%
where $\overline{z}_{k}$ is the standard complex conjugate of complex
numbers $z_{k}\in \mathbb{C}$.
\end{definition}

\begin{definition}[Bicomplex modulus]
\cite{6} Let $z_{1},z_{2}\in \mathbb{C}$ and $w=z_{1}+z_{2}\mathbf{i}_{2}\in 
\mathbb{T}$, then we have that:%
\begin{eqnarray*}
|w|_{\mathbf{i}_{1}}^{2} &=&w.w^{\dag _{2}}\in \mathbb{C}(\mathbf{i}_{_{1}})
\\
|w|_{\mathbf{i}_{2}}^{2} &=&w.w^{\dag _{1}}\in \mathbb{C}(\mathbf{i}_{_{2}})
\\
|w|_{\mathbf{j}}^{2} &=&w.w^{\dag _{3}}\in \mathbb{D}.
\end{eqnarray*}%
Also the \textit{Euclidean }$\mathbb{R}^{4}$\textit{-norm} defined as%
\begin{equation*}
\left\Vert w\right\Vert =\sqrt{|z_{1}|^{2}+|z_{2}|^{2}}=\sqrt{\func{Re}(|w|_{%
\mathbf{j}}^{2})}.
\end{equation*}
\end{definition}

\begin{definition}[Complex square norm]
\cite{9} The complex (square) norm $CN(w)$ of the bicomplex number $w$ is
the complex number $z_{1}^{2}+z_{2}^{2};$ as $w^{\dag _{2}}\mathbf{=}%
z_{1}-z_{2}\mathbf{i}_{2},$ we can see that $CN(w)=ww^{\dag _{2}}.$ Then a
bicomplex number $w=z_{1}+z_{2}\mathbf{i}_{2}$ is invertible if and only if $%
CN(w)\neq 0.$ Precisely,%
\begin{equation*}
w^{-1}=\frac{w^{\dag _{2}}}{CN(w)}\Rightarrow (z_{1}+z_{2}\mathbf{i}%
_{2})^{-1}=\frac{z_{1}}{z_{1}^{2}+z_{2}^{2}}-\frac{z_{2}}{z_{1}^{2}+z_{2}^{2}%
}\mathbf{i}_{2}.
\end{equation*}
\end{definition}

One should note that $CN(w)$\ is not a real number. Also $CN(w)=0$ if and
only if $z_{2}=\mathbf{i}_{1}z_{1}$\ or $z_{2}=-\mathbf{i}_{1}z_{1}$\ i.e,
if we consider $z_{1}$\ and $z_{2}$\ to be vector then they are of equal
length and perpendicular to each other.

\begin{definition}[Null cone]
\cite{11} The set of all zero divisors is called the \textit{null cone}.
That terminology comes from the fact that when $w$ is written as $z_{1}+z_{2}%
\mathbf{i}_{2},$ zero divisors are such that $z_{1}^{2}+z_{2}^{2}=0.$These
elements are also called singular elements, otherwise it is nonsingular. The
set of all \textit{singular elements} of $\mathbb{T}$\ is denoted by $%
\mathcal{NC}$ or $\mathcal{O}_{2}.$
\end{definition}

\begin{example}
The domain $D=\{(x_{1}+y_{1}\mathbf{i}_{1})+(x_{2}+y_{2}\mathbf{i}%
_{1}).i_{2}:x_{1}>0,y_{1}>0,x_{2}>0,y_{2}>0\}\subset \mathbb{T}$ is a
nonsingular domain.
\end{example}

\begin{definition}[Idempotent basis]
\cite{11} A bicomplex number $w=z_{1}+z_{2}\mathbf{i}_{2}$ has the following
unique idempotent representation:%
\begin{equation*}
w=z_{1}+z_{2}\mathbf{i}_{2}=(z_{1}-z_{2}\mathbf{i}_{1})\mathbf{e}%
_{1}+(z_{1}+z_{2}\mathbf{i}_{1})\mathbf{e}_{2}
\end{equation*}%
where $\mathbf{e}_{1}=\frac{1+\mathbf{j}}{2}$ and $\mathbf{e}_{2}=\frac{1-%
\mathbf{j}}{2}.$\newline
Also 
\begin{equation*}
x_{1}+x_{2}\mathbf{i}_{1}+x_{3}\mathbf{i}_{2}+x_{4}\mathbf{j}%
=((x_{1}+x_{4})+(x_{2}-x_{3}))\mathbf{e}_{1}+((x_{1}-x_{4})+(x_{2}+x_{3}))%
\mathbf{e}_{2}.
\end{equation*}
\end{definition}

\begin{definition}[Projection]
\cite{11} Two \textit{projection operators }$P_{1}$ and $P_{2}$ are defined
from $\mathbb{T}\longmapsto \mathbb{C}(\mathbf{i}_{_{1}})$ as%
\begin{equation*}
P_{1}(z_{1}+z_{2}\mathbf{i}_{2})=(z_{1}-z_{2}\mathbf{i}_{1})
\end{equation*}%
\begin{equation*}
P_{2}(z_{1}+z_{2}\mathbf{i}_{2})=(z_{1}+z_{2}\mathbf{i}_{1}).
\end{equation*}
\end{definition}

Thus a bicomplex number $w=z_{1}+z_{2}\mathbf{i}_{2}$ is written as 
\begin{equation*}
w=z_{1}+z_{2}\mathbf{i}_{2}=P_{1}(z_{1}+z_{2}\mathbf{i}_{2})\mathbf{e}%
_{1}+P_{2}(z_{1}+z_{2}\mathbf{i}_{2})\mathbf{e}_{2}=P_{1}(w)\mathbf{e}%
_{1}+P_{2}(w)\mathbf{e}_{2}.
\end{equation*}

This representation is very useful as addition, subtraction, multiplication
and division can be done term by term.

\begin{theorem}[The Extended bicomplex plane $\overline{\mathbb{T}}$]
\cite{11} We define the extended bicomplex plane $\overline{\mathbb{T}}$ as 
\begin{eqnarray*}
\overline{\mathbb{T}} &=&\overline{\mathbb{C}(\mathbf{i}_{1})}\times _{e}%
\overline{\mathbb{C}(\mathbf{i}_{1})} \\
&=&\left( \mathbb{C}(\mathbf{i}_{1})\cup \{\infty \}\right) \times
_{e}\left( \mathbb{C}(\mathbf{i}_{1})\cup \{\infty \}\right) \\
&=&\left( \mathbb{C}(\mathbf{i}_{1})\times _{e}\mathbb{C}(\mathbf{i}%
_{1})\right) \cup \left( \mathbb{C}(\mathbf{i}_{1})\times _{e}\{\infty
\}\right) \cup \left( \{\infty \}\times _{e}\mathbb{C}(\mathbf{i}%
_{1})\right) \cup \{\infty \} \\
&=&\mathbb{T\cup }\text{ }I_{\infty }. \\
I_{\infty } &=&\left( \mathbb{C}(\mathbf{i}_{1})\times _{e}\{\infty
\}\right) \cup \left( \{\infty \}\times _{e}\mathbb{C}(\mathbf{i}%
_{1})\right) \cup \{\infty \} \\
&=&\left\{ w\in \overline{\mathbb{T}}:\left\Vert w\right\Vert =\infty
\right\} .
\end{eqnarray*}
\end{theorem}

Clearly $\overline{\mathbb{T}}$\ is obtained by adding $\mathbb{T}$\ with an
infinity set $I_{\infty }$\ in stead of $\{\infty \}.$

\begin{definition}[$P_{1}$-Infinity and $P_{2}$-Infinity]
\cite{11} An element $w\in I_{\infty }$ is said to be a $P_{1}$-infinity ($%
P_{2}$-infinity) element if $P_{1}(w)=\infty $ $(P_{2}(w)=\infty )$ and $%
P_{2}(w)\neq \infty $ $(P_{1}(w)\neq \infty )$.
\end{definition}

\begin{definition}[$I_{1}$-Infinity Set and $I_{2}$-Infinity Set]
\cite{11} The set of all $P_{1}$-infinity elements is called the $I_{1}$%
-infinity set. It is denoted by $I_{1,\infty }$. Therefore,%
\begin{equation*}
I_{1,\infty }=\left\{ w\in \overline{\mathbb{T}}:P_{1}(w)=\infty
,P_{2}(w)\neq \infty \right\} .
\end{equation*}%
Similarly we can define the $I_{2}$-infinity set as:%
\begin{equation*}
I_{2,\infty }=\left\{ w\in \overline{\mathbb{T}}:P_{2}(w)=\infty
,P_{1}(w)\neq \infty \right\} .
\end{equation*}
\end{definition}

\begin{definition}[$P_{1}$-Zero and $P_{2}$-Zero]
\cite{11} An element $w\in \overline{\mathbb{T}}$ is said to be a $P_{1}$%
-zero ($P_{2}$-zero) element if $P_{1}(w)=0$ $(P_{2}(w)=0)$ and $%
P_{2}(w)\neq 0$ $(P_{1}(w)\neq 0)$.
\end{definition}

\begin{definition}[$I_{1}$-Zero set and $I_{2}$-Zero set]
\cite{11} The set of all $P_{1}$-zero elements is called the $I_{1}$-zero
set. It is denoted by $I_{1,0}$. Therefore,%
\begin{equation*}
I_{1,0}=\left\{ w\in \overline{\mathbb{T}}:P_{1}(w)=0,P_{2}(w)\neq 0\right\}
.
\end{equation*}%
Similarly we can define the $I_{2}$-zero set as:%
\begin{equation*}
I_{2,0}=\left\{ w\in \overline{\mathbb{T}}:P_{2}(w)=0,P_{1}(w)\neq 0\right\}
.
\end{equation*}
\end{definition}

\begin{definition}
\cite{11} We now construct the following two new sets:%
\begin{equation*}
I_{\infty }^{-}=I_{1,\infty }\cup I_{2,\infty }\text{ \ \ and \ \ }%
I_{0}^{-}=I_{1,0}\cup I_{2,0}
\end{equation*}%
so that 
\begin{equation*}
I_{\infty }=I_{\infty }^{-}\cup \left\{ \infty \right\} \text{ \ \ and \ \ }%
\mathcal{NC}=I_{0}^{-}\cup \left\{ 0\right\} .
\end{equation*}
\end{definition}

Each element in the null-cone has an inverse in $I_{\infty }$ and vice
versa. The elements of the set $I_{\infty }^{-}$ do not satisfy all the
properties as satisfied by the $\mathbb{C}(\mathbf{i}_{1})$-infinity but the
element $\infty =\infty \mathbf{e}_{1}+\infty \mathbf{e}_{2}$ does. We may
call the set $I_{\infty }^{-}$ , the\textbf{\ weak bicomplex infinity set}
and the element $\infty =\infty \mathbf{e}_{1}+\infty \mathbf{e}_{2}$, the 
\textbf{strong infinity}.

\begin{definition}[Fixed point]
Let $f:\overline{\mathbb{T}}\rightarrow \overline{\mathbb{T}},$ then $w\in 
\overline{\mathbb{T}}$ is called a fixed point of $f$ if and only if $%
f\left( w\right) =w.$
\end{definition}

\section{\textbf{Bicomplex Mobius Transformation}}

There is an amazing class of mapping defined on the extended complex plane $%
\mathbb{C}_{\infty }$\ called the Mobius transformation \cite{1a}. In this
section the bicomplex version of Mobius transformation is defined and some
interesting properties are studied.

\begin{definition}
Let $a,b,c,d\in \mathbb{T},$ $a,b,c,d$ be not divisors of zero and $%
ad-bc\notin \mathcal{NC},$ then the mapping $S:\overline{\mathbb{T}}%
\rightarrow \overline{\mathbb{T}}$ defined by%
\begin{equation*}
S\left( w\right) =\frac{aw+b}{cw+d}
\end{equation*}%
with the additional agreement $S\left( \infty \right) =\frac{a}{c}$ and $%
S\left( -\frac{d}{c}\right) =\infty $ is called a Bicomplex Mobius
Transformation or Bicomplex Mobius Map.
\end{definition}

Henceforth we shall simply write $S\left( w\right) =\frac{aw+b}{cw+d}$\ with
the tacit understanding that $S\left( \infty \right) =\frac{a}{c}$ and $%
S\left( -\frac{d}{c}\right) =\infty .$

Let $S\left( w\right) =\frac{aw+b}{cw+d}$\ and $\lambda \in \mathbb{T}$ and $%
\lambda \notin \mathcal{NC},$then\ 
\begin{equation*}
S\left( w\right) =\frac{\left( \lambda a\right) w+\left( \lambda b\right) }{%
\left( \lambda c\right) w+\left( \lambda d\right) }.
\end{equation*}%
That is the coefficients $a,b,c,d$ are not unique. Notice that we can not
have $a=0=c$ or $b=0=d$ since either situation would contradict $ad-bc%
\mathcal{\neq }0$.

\begin{example}
If $a\notin \mathcal{NC}$,then the mapping $aw$ from $\overline{\mathbb{T}}$
to $\overline{\mathbb{T}}$ is a bicomplex Mobius transformation which is
called \textit{dilation}.
\end{example}

\begin{example}
The mapping $\frac{1}{w}$ from $\overline{\mathbb{T}}$ to $\overline{\mathbb{%
T}}$ is a bicomplex Mobius transformation which is called \textit{inversion}.
\end{example}

\begin{example}
For all $b\in \mathbb{T},b$ is not a divisor of zero, the mapping $w+b$ from 
$\overline{\mathbb{T}}$ to $\overline{\mathbb{T}}$ is a bicomplex Mobius
transformation which is called \textit{translation}.
\end{example}

\begin{theorem}
Let $S_{1}\left( w\right) =\frac{a_{1}w+b_{1}}{c_{1}w+d_{1}}$ and $%
S_{2}\left( w\right) =\frac{a_{2}w+b_{2}}{c_{2}w+d_{2}}$ be two bicomplex
Mobius transformations. Then the composite function $S_{1}\circ S_{2}$ is
also a bicomplex Mobius transformation.
\end{theorem}

\begin{proof}
For all $w\in \mathbb{T}-\left\{ -\frac{d_{2}}{c_{2}}\right\} $ we have%
\begin{eqnarray*}
S_{1}\circ S_{2}\left( w\right) &=&S_{1}\left( S_{2}\left( w\right) \right)
\\
&=&\frac{a_{1}S_{2}\left( w\right) +b_{1}}{c_{1}S_{2}\left( w\right) +d_{1}}
\\
&=&\frac{a_{1}\frac{a_{2}w+b_{2}}{c_{2}w+d_{2}}+b_{1}}{c_{1}\frac{%
a_{2}w+b_{2}}{c_{2}w+d_{2}}+d_{1}} \\
&=&\frac{\left( a_{1}a_{2}+b_{1}c_{2}\right) w+\left(
a_{1}b_{2}+b_{1}d_{2}\right) }{\left( c_{1}a_{2}+d_{1}c_{2}\right) w+\left(
c_{1}b_{2}+d_{1}d_{2}\right) }.
\end{eqnarray*}

Put%
\begin{equation*}
M\left( w\right) =\frac{\left( a_{1}a_{2}+b_{1}c_{2}\right) w+\left(
a_{1}b_{2}+b_{1}d_{2}\right) }{\left( c_{1}a_{2}+d_{1}c_{2}\right) w+\left(
c_{1}b_{2}+d_{1}d_{2}\right) }.
\end{equation*}

Then obviously 
\begin{equation*}
M\left( w\right) =S_{1}\circ S_{2}\left( w\right)
\end{equation*}%
for all $w\in \mathbb{T}-\left\{ -\frac{d_{2}}{c_{2}}\right\} .$ Also%
\begin{eqnarray*}
&&\left( a_{1}a_{2}+b_{1}c_{2}\right) \left( c_{1}b_{2}+d_{1}d_{2}\right)
-\left( a_{1}b_{2}+b_{1}d_{2}\right) \left( c_{1}a_{2}+d_{1}c_{2}\right) \\
&=&\left\vert 
\begin{array}{cc}
a_{1} & b_{1} \\ 
c_{1} & d_{1}%
\end{array}%
\right\vert \left\vert 
\begin{array}{cc}
a_{2} & b_{2} \\ 
c_{2} & d_{2}%
\end{array}%
\right\vert \\
&=&\left( a_{1}d_{1}-b_{1}c_{1}\right) \left( a_{2}d_{2}-b_{2}c_{2}\right) \\
&\notin &\mathcal{NC}.
\end{eqnarray*}

Thus $M\left( w\right) $\ is a Bicomplex Mobius Transformation. Note that%
\begin{eqnarray*}
M\left( -\frac{d_{2}}{c_{2}}\right) &=&\frac{a_{1}}{c_{1}} \\
&=&S_{1}\left( I_{\infty }\right) \\
&=&S_{1}\left( S_{2}\left( -\frac{d_{2}}{c_{2}}\right) \right) \\
&=&S_{1}\circ S_{2}\left( -\frac{d_{2}}{c_{2}}\right) .
\end{eqnarray*}

Similarly it can be checked that 
\begin{equation*}
M\left( \mathcal{\infty }\right) =S_{1}\circ S_{2}\left( \infty \right) .
\end{equation*}

Thus 
\begin{equation*}
M\left( w\right) =S_{1}\circ S_{2}\left( w\right)
\end{equation*}%
for all $w\in \overline{\mathbb{T}}.$

Therefore $S_{1}\circ S_{2}$\ is a bicomplex Mobius tranformation.

\begin{theorem}
Let $G$\ be the set of all bicomplex Mobius transformations. Then $\left(
G,\circ \right) $ is a group.
\end{theorem}
\end{proof}

\begin{proof}
We have already seen that $G$ is closed under $\circ .$

$G$ contains an identity element $I\left( w\right) =w$ $\forall $ $w\in 
\overline{\mathbb{T}}$ for the operation $\circ .$

For all $S\left( w\right) =\frac{aw+b}{cw+d}\in G$ $\exists $ $S^{-1}\left(
w\right) =\frac{dw-b}{-cw+a}\in G.$

Also $\circ $\ is associative.

Therefore $\left( G,\circ \right) $ is a group.
\end{proof}

\begin{remark}
Let $S\left( w\right) =\frac{aw+b}{cw+d}\in G.$

Then 
\begin{eqnarray*}
S\left( w\right) &=&\frac{a}{c}+\frac{bc-ad}{c^{2}\left( z+\frac{d}{c}%
\right) } \\
&=&S_{4}\circ S_{3}\circ S_{2}\circ S_{1}\left( w\right)
\end{eqnarray*}%
for all $w\in \overline{\mathbb{T}},$ where 
\begin{equation*}
S_{1}\left( w\right) =w+\frac{d}{c},S_{1}\left( w\right) =\frac{1}{w}%
,S_{3}\left( w\right) =\frac{bc-ad}{c^{2}}w\text{ and }S_{4}\left( w\right)
=w+\frac{a}{c}.
\end{equation*}

This shows that the translations, the inversions \ and the dilations
generate the group of bicomplex Mobius transformation.
\end{remark}

A.A. Pogorui, R.M.R Dagnino \cite{4} have published a paper in $\left(
2006\right) $\ to find the zeros of a bicomplex polynomial. The proof of
that paper is better appreciated in the following application.

\begin{theorem}
The number of fixed points of a non-identity bicomplex Mobius transformation
is $1$ or $2$\ or $4.$
\end{theorem}

\begin{proof}
Let $S\left( w\right) =\frac{aw+b}{cw+d}$ be a Bicomplex Mobius
Transformation. Then the following cases arise.
\end{proof}

\textbf{Case 1}: If $c=0,a=d,b\neq 0$ then $S\left( w\right) =w+\frac{b}{d}.$

Then $w=\infty \mathbf{e}_{1}+\infty \mathbf{e}_{2}$ (strong infinity) is
the only fixed point of $S\left( w\right) $.

\textbf{Case 2}: If $c=0,a\neq d$ then $S\left( w\right) =\frac{a}{d}w+\frac{%
b}{d}.$

Then $w=\infty \mathbf{e}_{1}+\infty \mathbf{e}_{2}$ (strong infinity)\ is a
fixed point of $S\left( w\right) $.

Also $S\left( w\right) $ has another fixed point $w=\frac{b}{d\left( 1-\frac{%
a}{d}\right) }$ in $\mathbb{T}$.

\textbf{Case 3}: If $c\neq 0$, then the roots of the polynomial equation

\begin{equation}
cw^{2}+\left( d-a\right) w-b=0.  \label{2}
\end{equation}

are the fixed points of $S\left( w\right) .$

Let us decompose the coefficients and the variable with respect to the
idempotent elements $\mathbf{e}_{1}$ and $\mathbf{e}_{2}.$We have $w=w_{1}%
\mathbf{e}_{1}+w_{2}\mathbf{e}_{2},a=a_{1}\mathbf{e}_{1}+a_{2}\mathbf{e}%
_{2},b=b_{1}\mathbf{e}_{1}+b_{2}\mathbf{e}_{2},c=c_{1}\mathbf{e}_{1}+c_{2}%
\mathbf{e}_{2},d=d_{1}\mathbf{e}_{1}+d_{2}\mathbf{e}_{2}$ where $%
w_{1},w_{2},a_{1},a_{2},b_{1},b_{2},c_{1},c_{2},d_{1},d_{2}\in \mathbb{C}%
\left( \mathbf{i}_{1}\right) $ .

Equation $\left( \ref{2}\right) $ is reduced to the system 
\begin{eqnarray}
c_{1}w_{1}^{2}+\left( d_{1}-a_{1}\right) w_{1}-b_{1} &=&0  \label{4} \\
c_{2}w_{2}^{2}+\left( d_{2}-a_{2}\right) w_{2}-b_{2} &=&0.  \label{5}
\end{eqnarray}

The above two equations are polynomial equations with complex coefficients.

Then both the equations $\left( \ref{4}\right) $ and $\left( \ref{5}\right) $
have exactly two roots in $\mathbb{C}$ and thus equation $\left( \ref{2}%
\right) $ has four roots in $\mathbb{T}$.

\textbf{Case 4} : If $c=0,a=d,b=0$ then $S\left( w\right) =w$ for all $w\in 
\overline{\mathbb{T}}.$ Therefore $S\left( w\right) $\ is an identity
transformation. Hence all points $w\in \overline{\mathbb{T}}$ are the fixed
points of $S\left( w\right) .$

Hence the theorem is proved.

\begin{example}
$S\left( w\right) =w+\left\{ \left( 1+2\mathbf{i}_{1}\right) \mathbf{e}%
_{1}+\left( 1+3\mathbf{i}_{1}\right) \mathbf{e}_{2}\right\} $ have the only
fixed point $w=\infty \mathbf{e}_{1}+\infty \mathbf{e}_{2}.$
\end{example}

\begin{example}
$S\left( w\right) =\left\{ \left( 2+3\mathbf{i}_{1}\right) \mathbf{e}%
_{1}+\left( 1+4\mathbf{i}_{1}\right) \mathbf{e}_{2}\right\} w+\left\{ \left(
1+2\mathbf{i}_{1}\right) \mathbf{e}_{1}+\left( 1+3\mathbf{i}_{1}\right) 
\mathbf{e}_{2}\right\} $ have two fixed points $w=\infty \mathbf{e}%
_{1}+\infty \mathbf{e}_{2}$\ and $w=\frac{\left( 1+2\mathbf{i}_{1}\right) 
\mathbf{e}_{1}+\left( 1+3\mathbf{i}_{1}\right) \mathbf{e}_{2}}{1-\left\{
\left( 2+3\mathbf{i}_{1}\right) \mathbf{e}_{1}+\left( 1+4\mathbf{i}%
_{1}\right) \mathbf{e}_{2}\right\} }=\frac{\left( 1+2\mathbf{i}_{1}\right) 
\mathbf{e}_{1}+\left( 1+3\mathbf{i}_{1}\right) \mathbf{e}_{2}}{\left\{ 1%
\mathbf{e}_{1}+1\mathbf{e}_{2}\right\} -\left\{ \left( 2+3\mathbf{i}%
_{1}\right) \mathbf{e}_{1}+\left( 1+4\mathbf{i}_{1}\right) \mathbf{e}%
_{2}\right\} }=\frac{\left( 1+2\mathbf{i}_{1}\right) \mathbf{e}_{1}+\left(
1+3\mathbf{i}_{1}\right) \mathbf{e}_{2}}{\left\{ \left( -1+3\mathbf{i}%
_{1}\right) \mathbf{e}_{1}+\left( 4\mathbf{i}_{1}\right) \mathbf{e}%
_{2}\right\} }=\left( \frac{1}{2}-\frac{1}{2}\mathbf{i}_{1}\right) \mathbf{e}%
_{1}+\left( \frac{3}{4}-\frac{1}{4}\mathbf{i}_{1}\right) \mathbf{e}_{2}$.
\end{example}

\begin{example}
$S\left( w\right) =\frac{\left\{ \left( 4+5\mathbf{i}_{1}\right) \mathbf{e}%
_{1}+\left( 1+2\mathbf{i}_{1}\right) \mathbf{e}_{2}\right\} w-\left\{ \left(
-1+3\mathbf{i}_{1}\right) \mathbf{e}_{1}+\left( 1+8\mathbf{i}_{1}\right) 
\mathbf{e}_{2}\right\} }{\left\{ 1\mathbf{e}_{1}+1\mathbf{e}_{2}\right\}
w+\left\{ \left( 2+2\mathbf{i}_{1}\right) \mathbf{e}_{1}-\left( 3+2\mathbf{i}%
_{1}\right) \mathbf{e}_{2}\right\} }.$

To find the fixed points of $S\left( w\right) $\ we get the following
equation%
\begin{equation}
\left\{ 1\mathbf{e}_{1}+1\mathbf{e}_{2}\right\} w^{2}+\left\{ -\left( 2+3%
\mathbf{i}_{1}\right) \mathbf{e}_{1}-\left( 4+4\mathbf{i}_{1}\right) \mathbf{%
e}_{2}\right\} w+\left\{ \left( -1+3\mathbf{i}_{1}\right) \mathbf{e}%
_{1}+\left( 1+8\mathbf{i}_{1}\right) \mathbf{e}_{2}\right\} =0.  \label{1}
\end{equation}%
Equation $\left( \ref{1}\right) $ is reduced to the system of polynomial
equations with complex coefficients

\begin{eqnarray}
w_{1}^{2}-\left( 2+3\mathbf{i}_{1}\right) w_{1}+\left( -1+3\mathbf{i}%
_{1}\right) &=&0  \label{6} \\
w_{2}^{2}-\left( 4+4\mathbf{i}_{1}\right) w_{2}+\left( 1+8\mathbf{i}%
_{1}\right) &=&0.  \label{7}
\end{eqnarray}

From equation $\left( \ref{6}\right) $\ and $\left( \ref{7}\right) $\ we get 
$w_{1}=\left( 1+\mathbf{i}_{1}\right) ,\left( 1+2\mathbf{i}_{1}\right) $ and 
$w_{2}=\left( 2+\mathbf{i}_{1}\right) ,\left( 2+3\mathbf{i}_{1}\right) .$

Therefore the fixed points of $S\left( w\right) $\ are $w=\left( 1+\mathbf{i}%
_{1}\right) \mathbf{e}_{1}+\left( 2+\mathbf{i}_{1}\right) \mathbf{e}%
_{2},\left( 1+\mathbf{i}_{1}\right) \mathbf{e}_{1}+\left( 2+3\mathbf{i}%
_{1}\right) \mathbf{e}_{2},\left( 1+2\mathbf{i}_{1}\right) \mathbf{e}%
_{1}+\left( 2+\mathbf{i}_{1}\right) \mathbf{e}_{2},$ and $\left( 1+2\mathbf{i%
}_{1}\right) \mathbf{e}_{1}+\left( 2+3\mathbf{i}_{1}\right) \mathbf{e}_{2}.$
\end{example}

\end{document}